\documentclass{amsart}

\title[Incompressible hypersurfaces stably in 4-manifolds]{Stable existence of incompressible\\ 3-manifolds in 4-manifolds}

\author[Q Khan]{Qayum Khan}
\address{Department of Mathematics \hfill Indiana University \hfill Bloomington IN 47405 USA}
\email{qkhan@indiana.edu}

\author[G Smith]{Gerrit Smith}
\address{Department of Mathematics \hfill Iowa State University \hfill Ames IA 50011 USA}
\email{gnsmith@iastate.edu}

\usepackage{amsthm,amsmath,amssymb}
\usepackage[all]{xy}

\usepackage{xcolor}
\definecolor{dark-red}{rgb}{0.4,0.15,0.15}
\definecolor{dark-blue}{rgb}{0.15,0.15,0.4}
\definecolor{medium-blue}{rgb}{0,0,0.5}
\usepackage[colorlinks, linkcolor={dark-red}, citecolor={dark-blue}, urlcolor={medium-blue}]{hyperref} 

\newtheorem{thm}{Theorem}[section]
\newtheorem{cor}[thm]{Corollary}
\newtheorem{lem}[thm]{Lemma}
\newtheorem{prop}[thm]{Proposition}

\theoremstyle{definition}
\newtheorem{defn}[thm]{Definition}

\numberwithin{equation}{section}

\newcommand{\C}{\mathbb{C}}
\newcommand{\R}{\mathbb{R}}
\newcommand{\Z}{\mathbb{Z}}

\newcommand{\ko}{\mathfrak{o}}

\newcommand{\G}{\Gamma}

\newcommand{\Cl}{\mathrm{Cl}}
\newcommand{\const}{\mathrm{const}}

\newcommand{\Fr}{\mathrm{Fr}}
\newcommand{\hofib}{\mathrm{hofib}}

\newcommand{\id}{\mathrm{id}}
\newcommand{\Ker}{\mathrm{Ker}}
\newcommand{\proj}{\mathrm{proj}}
\newcommand{\Spin}{\mathrm{Spin}}
\newcommand{\Sq}{\mathrm{Sq}}

\newcommand{\bdry}{\partial}
\newcommand{\diffeo}{\thickapprox}
\newcommand{\iso}{\cong}
\newcommand{\longra}{\longrightarrow}

\newcommand{\gens}[1]{\left\langle #1 \right\rangle}
\newcommand{\Kronecker}[2]{\langle #1, #2 \rangle}
\newcommand{\ol}[1]{\overline{#1}}
\newcommand{\ul}[1]{\underline{#1}}
\newcommand{\wh}[1]{\widehat{#1}}
\newcommand{\wt}[1]{\widetilde{#1}}

\begin{document}

\begin{abstract}
Given a separating embedded connected 3-manifold in a closed 4-manifold, the Seifert--van Kampen theorem implies that the fundamental group of the 4-manifold is an amalgamated product along the fundamental group of the 3-manifold.
In the other direction, given a closed 4-manifold whose fundamental group admits an injective amalgamated product structure along the fundamental group of a 3-manifold, is there a corresponding geometric-topological decomposition of the 4-manifold in a stable sense?
We find an algebraic-topological splitting criterion in terms of the orientation classes and universal covers.
Also, we equivariantly generalize the Lickorish--Wallace theorem to regular covers.
\end{abstract}
\maketitle

\section{Introduction}

In this paper, we examine the correspondence between algebraic topology and the stable geometric topology of 4-dimensional manifolds.
Two 4-manifolds are \textbf{stably equivalent} if they become diffeomorphic after forming the connected sum with finitely many copies of $S^2 \times S^2$.
Note this does not change the fundamental group, signature, or spin of a 4-manifold, but does change the second Betti number.
As in stable homotopy theory, computations are more tractable and still distinguish many spaces.
The Whitney trick, which in higher dimensions allows for mirroring between topology and algebra, cannot be used in 4-manifolds since a disc may intersect itself.
By stabilization of the 4-manifold, self-intersections of a disc may be removed, by a modification called the Norman trick \cite[2.1]{CS}.

The Kneser conjecture in 3-manifold topology states that if the fundamental group of a closed 3-manifold $X$ is a free product $G_- * G_+$, then $X \cong X_- \# X_+$ where $X_{\pm}$ have fundamental group $G_{\pm}$ respectively.
This conjecture was proved by J~Stallings in his dissertation (see \cite[1.B.3, 2.B.3]{Stallings}).  Later, C~Feustel \cite{Feustel} and G\,A~Swarup \cite{Swarup} proved a generalized version of the conjecture when the fundamental group admits an injective amalgamated product along a surface group.

Following Hillman's work \cite{Hillman} on the 4-dimensional version of the Kneser conjecture, Kreck--L\"uck--Teichner proved the 4-dimensional conjecture is false \cite{KLT_counterexample} but is true if one allows for stabilization \cite{KLT}.
We investigate the problem of stably realizing injective amalgamated product decompositions of the fundamental group of a 4-manifold via separating embedded codimension-one submanifolds.

\subsection{Bistable results}

Our results on stable embeddings vary according to the so-called \emph{$w_2$-type} of the 4-manifold.
So we first consider a weaker equivalence relation.
We call 4-manifolds \textbf{bistably diffeomorphic} if they become diffeomorphic after connecting sum each with finitely many copies of the complex-projective plane $\C P_2$ (nonspin) and its orientation-reversal $\ol{\C P_2}$.
For any oriented 4-manifold $X$, denote
\[
X(r) ~:=~ X \# r(S^2 \times S^2) \qquad\qquad
X(a,b) ~:=~ X \# a(\C P_2) \# b(\ol{\C P_2}).
\]
Stable implies bistable, as $(S^2 \times S^2) \# (\C P_2) \diffeo 2 (\C P_2) \# (\ol{\C P_2})$ \cite[Cor~1,Lem~1]{Wall}.

Given a nonempty connected CW-complex $A$, by a continuous map $u: A \longra B\G$ \textbf{classifying the universal cover} $\wt{A}$, we mean the induced map $u_\#$ on fundamental groups is an isomorphism, for some basepoints.
The map $u$ is uniquely determined up to homotopy and composition with self-homotopy equivalences $B\alpha: B\G \longra B\G$ for $\alpha$ an automorphism of $\G$.
By a connected subcomplex being \textbf{incompressible}, we shall mean that the inclusion induces a monomorphism on fundamental groups.

\begin{thm}\label{thm:main}
Let $X$ be a oriented closed smooth 4-manifold.
Let $c: X \longra BG$ classify its universal cover.
Let $X_0$ be a connected oriented closed 3-manifold with fundamental group $G_0$.
Suppose $G = G_- *_{G_0} G_+$ with $G_0 \subset G_\pm$.
There exists an incompressible embedding of $X_0$ in some bistabilization $X(a,b)$ inducing the given injective amalgamation of fundamental groups,
if and only if there exists a map $d: X_0 \longra BG_0$ classifying its universal cover and satisfying the equation
\begin{equation}\label{eqn:main}
d_*[X_0] ~=~ \bdry c_*[X] ~\in~ H_3(G_0;\Z),
\end{equation}
with $\bdry$ the boundary in a Mayer--Vietoris sequence in group homology~\cite[III:6a]{Brown}.
\end{thm}

The simplest case of $G_0=1$ was done transparently by J~Hillman \cite{Hillman}, whose hands-on approach with direct manipulation of handles we generalize in this paper.

\begin{cor}[Hillman]\label{cor:Hillman}
Let $X$ be a connected orientable closed smooth 4-manifold whose fundamental group is a free product $G_- * G_+$.
Some bistabilization $X(a,b)$ is diffeomorphic to a connected sum $X_- \# X_+$ with $X_\pm$ having fundamental group $G_\pm$ respectively.
\end{cor}

Similarly, when $G_0=\Z$, note $X$ is bistably diffeomorphic to some $X_- \cup_{S^1 \times S^2} X_+$.

\begin{proof}
Here $G_0=1$, hence $H_3(G_0) = 0$.
Take $X_0 = S^3$ and $d$ the constant map.
\end{proof}

The proof of Theorem~\ref{thm:main} generalizes Hillman's strategy for proving Corollary~\ref{cor:Hillman} and employs an equivariant generalization of the Lickorish--Wallace theorem (\S\ref{sec:LWmaps}).
Theorem~\ref{thm:LWmaps} is a bordism version that slides 1-handles then does Wallace's trick.
Wallace's proof of Corollary~\ref{cor:LW} relied upon the Rohlin--Thom theorem ($\Omega^{SO}_3=0$).

\subsection{Stable results}

Shortly after Hillman's result, Kreck--L\"uck--Teichner offered an alternative proof, using Kreck's machinery of modified surgery theory \cite{Kreck}.
They were able to replace bistabilization with stabilization, due to a careful analysis of $w_2$-types and triviality of 3-plane bundles over embedded 2-spheres in certain 5-dimensional cobordisms \cite{KLT}.
In general, stabilization is required \cite{KLT_counterexample}.
Regarding the removal of $S^2 \times S^2$ factors (destabilization), see \cite{HK} and \cite{Khan}.

Recall that a \textbf{(stable) spin structure} $s$ on a smooth oriented manifold $M$ is a homotopy-commutative diagram (reduction of structure groups in \cite[II:1.3]{LM}):
\[\xymatrix{
& B\Spin \ar[d]\\
M \ar[r]_-{\tau_M} \ar[ur]^-{s} & BSO.
}\]

\begin{thm}[totally nonspin]\label{thm:totallynonspin}
Let $X$ be a oriented closed smooth 4-manifold whose universal cover has no spin structure.
Let $c: X \longra BG$ classify this cover.
Let $X_0$ be a connected oriented closed 3-manifold with fundamental group $G_0$.
Suppose $G = G_- *_{G_0} G_+$ with $G_0 \subset G_\pm$.
There exists an incompressible embedding of $X_0$ in some $X(r)$ inducing the given injective amalgam of fundamental groups,
if and only if there is $d: X_0 \longra BG_0$ classifying its universal cover satisfying \eqref{eqn:main}.
\end{thm}

Any oriented manifold has a spin structure if and only if $w_2$ of its tangent bundle vanishes \cite{LM}.
So any oriented 3-manifold has a spin structure (as $w_2=v_2=0$).

\begin{thm}[spinnable]\label{thm:spin}
Let $X$ be a oriented closed smooth 4-manifold that admits some spin structure.
Let $c: X \longra BG$ classify the universal cover.
Let $X_0$ be a connected oriented closed 3-manifold with fundamental group $G_0$.
Suppose $G = G_- *_{G_0} G_+$ with $G_0 \subset G_\pm$.
There exists an incompressible embedding of $X_0$ in a stabilization $X(r)$ inducing the given injective amalgam of fundamental groups,
if and only if there exist a map $d: X_0 \longra BG_0$ classifying its universal cover and spin structures $s$ on $X$ and $t$ on $X_0$ satisfying:
\begin{equation}\label{eqn:spin}
[X_0,t,d] ~=~ \bdry [X,s,c] ~\in~ \Omega^\Spin_3(BG_0),
\end{equation}
with $\bdry$ the boundary map in a Mayer--Vietoris sequence in spin bordism \cite[5.7]{CF}.
\end{thm}

Observe that \eqref{eqn:spin} is a lift of \eqref{eqn:main}, via the cobordism-Hurewicz homomorphism
\[
\Omega^\Spin_3(BG_0) \xrightarrow{~epi~} \Omega^{SO}_3(BG_0) \xrightarrow{~iso~} H_3(BG_0).
\]

Finally, we generalize Theorem~\ref{thm:spin} to only require that $\wt{X}$ admits a spin structure.
In order to understand the more delicate criterion, we state a lemma and definition.

\begin{lem}\label{lem:prespin}
Let $u: Y \longra B\G$ classify the universal cover of an oriented connected smooth manifold $Y$.
The universal cover $\wt{Y}$ admits a spin structure if and only if there is a class $w_2^u \in H^2(\G;\Z/2)$ satisfying the equation $w_2(TY) = u^*(w_2^u)$.
When such a class exists it is unique.
\end{lem}

The secondary characteristic class $w_2^u$ will vanish if $Y$ admits a spin structure.
The following definition is rather delicate due to two explicit choices of homotopies.
For $H: A \times [0,1] \longra B$ and $a \in A$, the \textbf{$a$-track} is $H^a := (t \longmapsto H(a,t)) \in B^{[0,1]}$.

\begin{defn}\label{defn:inducedspin}
Let $Y$ be an oriented connected smooth manifold whose universal cover $\wt{Y}$ admits a spin structure.
Let $u: Y \longra B\G$ classify the universal cover.
Fix homotopy representatives $\tau_Y : Y \longra BSO, w_2 : BSO \longra K(\Z/2,2), w_2^u : B\G \longra K(\Z/2,2)$.
By Lemma~\ref{lem:prespin}, there is a homotopy $\eta$ from $w_2 \circ \tau_Y$ to $w_2^u \circ u$.
Suppose $\G = \G_- *_{\G_0} \G_+$ with $\G_0 \subset \G_\pm$.
Write $i_0: \G_0 \longra \G$ for the inclusion homomorphism.
Assume $Bi_0: B\G_0 \longra B\G$ is the inclusion of a bicollared subspace, with $u$ transverse to $B\G_0$.
If there exists a nulhomotopy $\theta$ of the map $w_2^u \circ Bi_0$, then we define the \emph{induced spin structure} $s_\eta^\theta$ on the submanifold $N := u^{-1}(B\G_0)$ of $Y$ by
\[
s_\eta^\theta : N \longra B\Spin ~;~ x \longmapsto \left( \tau_Y(x), \eta^{x} * \theta^{u(x)} \right),
\]
where we identify $B\Spin$ with the homotopy fiber of $w_2$ and $*$ denotes join of paths.
\end{defn}

We arrive at a generalization of Theorem~\ref{thm:spin} which further requires $i_0^*(w_2^c)=0$.

\begin{thm}[pre-spinnable]\label{thm:prespin}
Let $X$ be a oriented closed smooth 4-manifold whose universal cover admits a spin structure.
Let $c: X \longra BG$ classify this cover.
Let $X_0$ be a connected oriented closed 3-manifold with fundamental group $G_0$.
Suppose $G = G_- *_{G_0} G_+$ with $G_0 \subset G_\pm$.
There exists an incompressible embedding of $X_0$ in some $X(r)$ inducing the given injective amalgam of fundamental groups,
if and only if there exist a map $d: X_0 \longra BG_0$ classifying its universal cover and a spin structure $t$ on $X_0$ and a nulhomotopy $\theta$ of $w_2^c \circ Bi_0$ satisfying, with $M := c^{-1}(BG_0)$:
\begin{equation}\label{eqn:prespin}
[X_0,t,d] ~=~ \left[ M,s_\eta^\theta,c|M \right] ~\in~ \Omega^\Spin_3(BG_0).
\end{equation}
\end{thm}

The special case \cite{KLT} is a consequence of Theorems~\ref{thm:totallynonspin} and \ref{thm:prespin}.

\begin{cor}[Kreck--L\"uck--Teichner]\label{cor:KLT}
Let $X$ be a nonempty connected orientable closed smooth 4-manifold whose fundamental group is a free product $G_- * G_+$.
Some $X(r)$ is diffeomorphic to a sum $X_- \# X_+$ with each $X_\pm$ of fundamental group $G_\pm$.
\end{cor}

\begin{proof}
Here $G_0=1$, so $H_3(G_0) = 0 = \Omega^\Spin_3(BG_0)$.
Take $X_0 = S^3$, $d$ constant.
\end{proof}

Albeit that Kreck's modified surgery theory \cite{Kreck} is a powerful formalism, by which we were inspired and against which we checked our progress, we sought to write this paper from first principles, to be accessible to low-dimensional topologists.
In particular, we avoid `subtraction of solid tori' and `stable $s$-cobordism theorem.'

\section{Surgery on a link and regular covers}\label{sec:LWmaps}

We generalize the notion of classifying a universal cover.
For a nonempty connected CW-complex $A$, a continuous map $u: A \longra B\G$ \textbf{classifies a regular cover} means that the induced map $u_\#$ on fundamental groups is an epimorphism, for a choice of basepoints.
The (connected) regular cover $\wh{A}$ corresponds to the kernel of $u_\#$, and its covering group is identified with $\G$, which acts transitively on the fibers.

\subsection{Oriented version}
This development is used to prove Theorems~\ref{thm:main} and \ref{thm:totallynonspin}.

\begin{thm}\label{thm:LWmaps}
Let $M$ and $M'$ be connected oriented closed 3-manifolds.
Let $f: M \longra B\G$ and $f': M' \longra B\G$ classify regular covers.
Then there exists a framed oriented link $L$ in $M$ that transforms $(M,f)$ into $(M',f')$ by surgery if and only if
\begin{equation}\label{eqn:LWmaps}
f_*[M] ~=~ f'_*[M'] ~\in~ H_3(\G;\Z).
\end{equation}
\end{thm}

In other words, this is an algebraic-topological criterion for whether or not there is a link in $M$ whose preimage in $\wh{M}$ has a $\G$-equivariant surgery resulting in $\wh{M'}$.

The original version is simply without reference maps; see \cite{Wallace} and \cite{Lickorish}.

\begin{cor}[Lickorish--Wallace]\label{cor:LW}
Any nonempty connected oriented closed 3-manifold $N$ is the result of surgery on a framed oriented link $L$ in the 3-sphere.
\end{cor}

Lickorish also obtained each component is unknotted with $\pm 1$ Dehn coefficients.

\begin{proof}
Here $\G=1, M=S^3, M'=N$.
Note $B\G$ is a point, hence $H_3(\G)=0$.
\end{proof}

Here is a more general, technical version of Theorem~\ref{thm:LWmaps} that we shall use later.

\begin{lem}\label{lem:LWmaps}
Let $M$ and $M'$ be connected oriented closed 3-manifolds, and let $B$ be a connected CW-complex.
Suppose $f: M \longra B$ and $f': M' \longra B$ are continuous maps that induce epimorphisms on fundamental groups, for some basepoints.
Then
\begin{equation}\label{eqn:LWmaps_B}
f_*[M] ~=~ f'_*[M'] ~\in~ H_3(B;\Z) 
\end{equation}
if and only if there is a 4-dimensional smooth connected oriented compact bordism
\[
(F;f,f') ~:~ (W;M,M') ~\longra~ (B \times [0,1]; B \times \{0\}, B \times \{1\})
\]
such that $W$ has no 1-handles with respect to $M$ and no 1-handles with respect to $M'$, for a certain handle decomposition of the 4-dimensional cobordism $(W;M,M')$.
\end{lem}

\begin{proof}[Proof of Theorem~\ref{thm:LWmaps}]
By Lemma~\ref{lem:LWmaps}, use only 2-handles: surger along a link.
\end{proof}

The argument below, after the preliminary three paragraphs, can be perceived in two geometric steps, even though it is combined into a single surgical move.
The first step is to \emph{slide 1-handles}, along with the map data, so that they become trivial.
The second step is a \emph{reference-maps version} of Wallace's trick to exchange oriented 1-handles for trivial 2-handles \cite[5.1]{Wallace}.
(If $\dim M > 3$, see \cite[6.15]{RS} and subsequent remark to replace 1-handles for 3-handles in certain cobordisms on $M$.)

\begin{proof}[Proof of Lemma~\ref{lem:LWmaps}]
$\Longleftarrow$ is due to $\Omega^{SO}_3(B) \iso H_3(B)$.
Consider the $\Longrightarrow$ direction.

Clearly $\Omega^{SO}_0=\Z$ and $\Omega^{SO}_1=\Omega^{SO}_2=0$; recall that $\Omega^{SO}_3=0$ by Rohlin--Thom \cite[IV.13]{Thom}.
Then note, for the CW-complex $B$, by the Atiyah--Hirzebruch spectral sequence, that the cobordism-Hurewicz map is an isomorphism:
\[
\Omega^{SO}_3(B) ~\longra~ H_3(B) \quad ; \quad [M,f: M \to B] ~\longmapsto~ f_*[M].
\]
Thus the criterion \eqref{eqn:LWmaps_B} transforms into the equation: $[M,f] = [M',f'] \in \Omega^{SO}_3(B)$.
In other words, there exists a 4-dimensional smooth oriented compact bordism
\[
(F_0;f,f') ~:~ (W_0;M,M') ~\longra~ (B \times [0,1]; B \times \{0\}, B \times \{1\}).
\]

Since $M$ and $M'$ are connected, by joining their two possibly different components in $W_0$ via connected sum and ignoring the rest, we may assume that $W_0$ is connected.
Hence, in the handle decomposition of a Morse function $(W_0;M,M') \longra ([0,1];\{0\},\{1\})$, $W_0$ has no 0-handles with respect to $M$ and no 0-handles with respect to $M'$.
Therefore, it remains to eliminate the 1-handles of $W_0$ with respect to $M$ and $M'$.
For simplicity of notation, we assume that $W_0$ has a single 1-handle.

Let $h: (D^1 \times D^3, S^0 \times D^3) \longra (W_0, M)$ be the 1-handle, preserving orientation.
Since $M$ is connected, there is a path $\alpha_0: [-1,1] \longra M$ with $\alpha_0(\pm 1) = h(\mp 1,0)$.
Concatenation yields a loop $\beta_0 = h(-,0) * \alpha_0: S^1 \longra W_0$.
Since $f_\#$ is an epimorphism, there exists a loop $\beta: (S^1,1) \longra (M,\alpha_0(1))$ such that $f_\#[\beta]=F_{0\#}[\beta_0]$.
By general position, there is a normally framed embedded arc $\alpha_1: [-1,1] \times D^2 \longra M$ such that $\alpha_1(\pm 1,0) = h(\mp 1,0)$ and $\alpha_1(-,0)$ is homotopic rel boundary to $\alpha_0 * \beta^{-1}$.

Push off the 1-handle core $h(-,0)$ to obtain a normally framed embedded arc $h_0: [-1,1] \times D^2 \longra M'$.
A matching isotopy takes $\alpha_1$ to $\alpha'_1: [-1,1] \times D^2 \longra M'$ with $\alpha'_1(\pm 1) = h_1(\mp 1, 0)$.
Concatenation yields a normally framed embedded loop
\[
\lambda ~:=~ \{h_0(-,r) * \alpha'_1(-,r) \}_{r \in D^2} ~:~ S^1 \times D^2 \longra M'.
\]
Write $W_1 := M' \times [0,1] \cup_\lambda D^2 \times D^2$ for the trace of the surgery along $\lambda$ in $M'$.
Since $F_0 \circ \lambda(-,0)$ is nulhomotopic, choose a nulhomotopy to yield a bordism
\[
(F_1;f',f'') : (W_1;M',M'') ~\longra~ (B \times [0,1]; B \times \{0\}, B \times \{1\}).
\]
Observe that $\ol{W_1}$ is the trace of surgery along the framed belt sphere $S^1 \times D^2 \hookrightarrow M''$.
By the cancellation lemma \cite[6.4]{RS}, $W_0 \cup_{M'} W_1$ is diffeomorphic to $M \times [0,2]$ relative to $M \times \{0\}$.
In particular, there is $\delta: M \diffeo M''$ with $f'' \circ \delta \simeq f$.
Write $W_0' := M \times [0,2] \cup_\delta \ol{W_1}$.
So we have a new bordism with $h$ replaced by a 2-handle:
\[
(\delta \cup_{f''} F_1; f, f') : (W_0'; M, M') ~\longra~ (B \times [0,1]; B \times \{0\}, B \times \{1\}).
\]

By iteration, we kill all 1-handles of $W_0'$ relative to $M$.
Similarly, repeat relative to $M'$.
Thus, we obtain the desired bordism $(W,F)$ with only 2-handles rel $M$.
\end{proof}

\subsection{Spin version}
We shall need this development to prove Theorem \ref{thm:spin}.

\begin{thm}\label{thm:LWmaps_spin}
Let $(M,s)$ and $(M',s')$ be spin closed 3-manifolds.
Let $f: M \longra B\G$ and $f': M' \longra B\G$ classify regular covers.
There exists a framed oriented link $L$ in $M$ that transforms $(M,s,f)$ into $(M',s',f')$ by a spin bordism if and only if
\begin{equation}\label{eqn:LWmaps_spin}
[M,s,f] ~=~ [M',s',f'] ~\in~ \Omega^\Spin_3(B\G).
\end{equation}
\end{thm}

\begin{lem}\label{lem:LWmaps_spin}
Let $(M,s)$ and $(M',s')$ be connected spin closed 3-manifolds, and let $B$ be a connected CW-complex.
Suppose $f: M \longra B$ and $f': M' \longra B$ are continuous maps that induce epimorphisms on fundamental groups.
Then
\begin{equation}\label{eqn:LWmaps_B_spin}
[M,s,f] = [M',s',f'] \in \Omega^\Spin_3(B) 
\end{equation}
if and only if there is a 4-dimensional smooth connected spin compact bordism
\[
(F;f,f') ~:~ (W,t;M,s,M',s') ~\longra~ (B \times [0,1]; B \times \{0\}, B \times \{1\})
\]
such that $W$ has no 1-handles with respect to $M$ and no 1-handles with respect to $M'$, for a certain handle decomposition of the 4-dimensional cobordism $(W;M,M')$.
\end{lem}

\begin{proof}[Proof of Theorem~\ref{thm:LWmaps_spin}]
By Lemma~\ref{lem:LWmaps_spin}, use only 2-handles: surger along a link.
\end{proof}

\begin{proof}[Proof of Lemma~\ref{lem:LWmaps_spin}]
The $\Longleftarrow$ implication is obvious.
Consider the $\Longrightarrow$ implication. 

Recall the proof of Lemma~\ref{lem:LWmaps}.
We reconstruct $W_1$ to admit a spin structure extending the spin structure $\ol{s'}$ on $\bdry_-W_1 = \ol{M'}$, since by gluing along $M''$ this will induce a spin structure on $W_0'$ extending the spin structure $\ol{s} \sqcup s'$ on $\bdry W_0' = \ol{M} \sqcup M'$.

Since $H^{i+1}(M';\pi_i(\Spin_3))=0$ for all $i \geqslant 0$, by obstruction theory, the spin structure $s'$ lifts to a framing $\phi$ of the tangent bundle $TM'$.
The sole obstruction to extending the stable framing $\phi\oplus\id$ of $TM'\oplus \ul{\R}$ to the tangent bundle $\tau$ of $W_1$ is
\[
\ko(\tau) ~\in~ H^2(W_1,M'; \pi_1(SO_4)) ~~=~ H^2(D^2,S^1; \pi_1(SO_4)) ~=~ \pi_1(SO_4) ~\iso~ \Z/2.
\]
Let $\eta \in \pi_1(SO_2) \iso \Z$.
Reframe the normal bundle of the surgery circle $\lambda(-,0)$ as
\[
\lambda^\eta ~:~ S^1 \times D^2 \longra M' ~;~ (z,r) \longmapsto \lambda(z,\eta_z(r)).
\]
Write $W_1^\eta := M' \times [0,1] \cup_{\lambda^\eta} D^2 \times D^2$ with tangent bundle $\tau^\eta$.
By \cite[Lemma~6.1]{KM},
\[
\ko(\tau^\eta) ~=~ \ko(\tau) + \sigma_\#(\eta) ~\in~ \pi_1(SO_4),
\]
where $\sigma: SO_2 \longra SO_4$ denotes the inclusion.
Since the induced map $\sigma_\#$ on fundamental groups is surjective, find $\eta$ so that $\tau^\eta$ has a framing extending $\phi\oplus\id$.
Hence $W_1^\eta$ has a spin structure extending $\ol{s'}$ on its lower oriented boundary $\ol{M'}$.

By gluing, we obtain an induced spin structure on $W_0 \cup_{M'} W_1^\eta \diffeo M \times [0,2]$ relative to $M \times \{0\}$.
Modifying Proof~\ref{lem:LWmaps}, redefine $W_0' := M \times [0,2] \cup_{\delta^\eta} \ol{W_1^\eta}$ with spin structure the union of this one and the orientation-reversal of the one on $W_1^\eta$.
Therefore, the spin structure on $W_0'$ restricts to $\ol{s} \sqcup s'$ on $\bdry W_0' = \ol{M} \sqcup M'$.
\end{proof}

\section{Ambient surgery on pairs of points}

Let $G = G_- *_{G_0} G_+$ be an injective amalgam of groups.
The corresponding \textbf{double mapping cylinder model} of its classifying space is the homotopy colimit
\begin{equation}\label{eqn:dmcm}
BG ~:=~ BG_- \;\cup_{BG_0 \times \{-1\}}\; BG_0 \times [-1,+1] \;\cup_{BG_0 \times \{+1\}}\; BG_+
\end{equation}
with respect to the maps $BG_0 \longra BG_\pm$ induced from the inclusions $G_0 \longra G_\pm$.

Akin to Stalling's thesis, here is a folklore fact proven in \cite[1.1]{Bowditch} (cf.~\cite{BS}).

\begin{thm}[Bowditch]\label{thm:Bowditch}
If $G$ and $G_0$ are finitely presented, so are $G_-$ and $G_+$.
\end{thm}

Instead of $G_0$ being finitely presented, the proof of the next statement can work assuming $G_-, G_0, G_+$ are finitely generated, but we prefer the former hypothesis.

\begin{prop}\label{prop:epi}
Let $X$ be a connected oriented closed smooth 4-manifold.
Suppose $f: X \longra BG$ classifies a regular cover.
Assume $G_0$ is finitely presented.
Then $f$ can be re-chosen up to homotopy so that: $f$ is transverse to the bicollared subspace $BG_0 \times \{0\}$ in the model \eqref{eqn:dmcm}, the 3-submanifold preimage $M$ in $X$ is connected, and the restriction $f: M \longra BG_0$ also classifies a regular cover.
\end{prop}

This is proven after three lemmas.
The first is an apparatus to recalibrate paths.

\begin{lem}\label{lem:loops}
Let $X$ be a connected oriented smooth $n$-manifold with $n>2$.
Consider a space $B=B_- \cup_{B_0} B_+$ with $B,B_\pm - B_0$ path-connected and $B_0 = B_- \cap B_+$.
Suppose $f: X \longra B$ is $\pi_1$-surjective.
Assume $\pi_1(B_\pm-B_0)$ are finitely generated, by $r_\pm$ elements.
There are disjoint 1-handlebodies $\Lambda_\pm \diffeo \# r_\pm(S^1 \times D^{n-1}) \subset X$ and $f': X \longra B$ homotopic to $f$ having $\pi_1$-surjective restrictions $f': \Lambda_\pm \longra B_\pm-B_0$.
\end{lem}

\begin{proof}
We may homotope $f$ so that its image contains some points $b_\pm$ in $B_\pm-B_0$.
Then there are $x_\pm \in X$ such that $f(x_\pm)=b_\pm$.
There are based loops $\mu^1_\pm, \ldots, \mu^{r_\pm}_\pm: (S^1,1) \longra (B_\pm-B_0,b_\pm)$ whose based homotopy classes generate $\pi_1(B_\pm-B_0,b_\pm)$.
Since $f_\#: \pi_1(X,x_\pm) \longra \pi_1(B,b_\pm)$ is surjective and $n>2$, there exist disjoint smoothly embedded based loops $\lambda^1_\pm, \ldots, \lambda^{r_\pm}_\pm: (S^1,1) \longra (X,x_\pm)$ and a based homotopy $H^i_\pm: S^1 \times [0,1] \longra B$ from $f \circ \lambda^i_\pm$ to $\mu^i_\pm$.
Since $X$ is oriented, for each $i$, there is a tubular neighborhood $\Lambda_i^\pm$ of $\lambda_i^\pm(S^1)$ and a diffeomorphism $\Lambda_i^\pm \diffeo S^1 \times D^{n-1}$.
Taking the radii of the tubes sufficiently small, we find that the $\Lambda_i^\pm$ pairwise intersect in a fixed $D^n$-neighborhood of $x_\pm$.
Thus we obtain disjoint embeddings of $\# r_-(S^1 \times D^{n-1})$ and $\# r_+(S^1 \times D^{n-1})$ in $X$, say with images called $\Lambda_-$ and $\Lambda_+$.

Finally, using these NDR neighborhoods $\Lambda_\pm$ of $\bigvee_i \lambda^i_\pm$ and specific homotopies $\bigvee_i H^i_\pm$ as the data for the homotopy extension property \cite[Theorem~VII:1.5]{Bredon}, we obtain a homotopy $H: X \times [0,1] \longra B$ from $f$ to a map $f'$ such that $f'\circ \lambda^i_\pm = \mu^i_\pm$.
Hence $f'(\Lambda_\pm) \subset B_\pm-B_0$ and $(f'|\Lambda_\pm)_\#\pi_1(\Lambda_\pm,x_\pm) = \pi_1(B_\pm-B_0,b_\pm)$.
\end{proof}

Given a continuous map $f: X \longra B$ from a smooth manifold $X$ to a topological space $B$, and given a subspace $B_0$ that admits a tubular neighborhood $E(\xi) \subset B$, W~Browder defines \textbf{$f$ to be transverse to $B_0$} to mean that the conclusion of the implicit-function theorem holds:
the preimage $X_0 = f^{-1}(B_0)$ is a smooth submanifold of $X$ with normal bundle $\nu(X_0 \hookrightarrow X) = (f|X_0)^*(\xi)$ \cite[II:\S2]{Browder}.
Since the proof is omitted for Browder's generalization \cite[II:2.1]{Browder} of Thom's transversality theorem \cite[I:5]{Thom}, we give details for the trivial line bundle $\xi = \ul{\R}$.

\begin{lem}\label{lem:transverse}
Let $f: X \longra B$ be a continuous map from a smooth manifold to a space $B$.
For any \textbf{bicollared subspace} $B_0$ of $B$ (i.e., $B_0$ has a neighborhood in $B$ homeomorphic to $B_0 \times \R$), there exists a map $f': X \longra B$ transverse to $B_0$ and homotopic to $f$, relative to the complement of an open neighborhood of $f^{-1}(B_0)$.
\end{lem}

\begin{proof}
We have an open embedding $\beta: B_0 \times \R \longra B$ with $\beta(B_0 \times \{0\}) = B_0 \subset B$.
Write $N \subset B$ for the image of $\beta$, and write $\pi_2: B_0 \times \R \longra \R$ for the projection.
Note $f^{-1}(N) \subset X$ is a smooth manifold, since it is an open set in a smooth manifold.
By Whitney's approximation theorem \cite[II:11.7]{Bredon}, the $C^0$ function $\pi_2 \circ \beta^{-1} \circ f: f^{-1}(N) \longra \R$ is $0.5$-close to a $C^\infty$ function $g: f^{-1}(N) \longra \R$.
Define
\[
H : f^{-1}(N) \times [0,1] \longra B ~;~ (x,t) \longmapsto \beta\left( (\pi_1 \beta^{-1} f)(x), (1-t) (\pi_2 \beta^{-1} f)(x) + t g(x) \right).
\]
Note $H$ is a homotopy from $H(x,0)=f(x)$ to a map $f' := H(-,1) : f^{-1}(N) \longra B$.
Then $f'$ is transverse to $B_0$ with $(f')^{-1}(B_0)=g^{-1}\{0\}$ a smooth submanifold of $X$; where by Sard's theorem and a tiny homotopy, we assume $0$ is a regular value of $g$.

It remains to extend $H$ to $X \times [0,1]$ so that $H(x,t)=f(x)$ for all $x \in X-N$.
Using explicit formulas derived from the tubular neighborhood structure $\beta$, this is achieved by the homotopy extension property for the neighborhood deformation retract $f^{-1}\beta(B_0 \times [-1,1]) \sqcup (X-N)$ closed in the $T_4$ space $X$; see \cite[Theorem~VII:1.5]{Bredon}.
The desired map $f': X \longra B$ is again $H(-,1)$ of this extension.
\end{proof}

We perform \emph{1-handle exchanges in dimension 4} by an obstruction-theoretic argument.
This is not in the literature, but see \cite[p67]{Hempel} and \cite[Lemma~I:3]{Cappell}.
Recall the \textbf{frontier} $\Fr_X(A) := \Cl_X(A)\cap\Cl_X(X-A)$ for $A \subset X$, a topological space.

\begin{lem}\label{lem:handleexchange}
Let $f: X \longra B$ be a continuous map from a smooth 4-manifold to a path-connected space $B$, transverse to a path-connected separating subspace $B_0$ of $B = B_- \cup_{B_0} B_+$.
Decompose $X = X_- \cup_{X_0} X_+$ by the $f$-preimages.
Let $\alpha: (D^1,\bdry D^1) \longra (X_\pm,X_0)$ be a smoothly embedded arc with $[f\circ \alpha] = 0 \in \pi_1(B_\pm,B_0)$.
Suppose $\pi_3(B_\mp) = 0 = \pi_4(B)$.
Then $f$ is homotopic to a $B_0$-transverse map $g: X \longra B$ whose preimage of $B_0$ is the result of adding a 1-handle with core $\alpha$.
Namely, for some open-tubular neighborhood $U \diffeo D^1 \times \mathring{D}^3$ of the arc $\alpha$ in $X_\pm$:
\[
g^{-1}(B_0) = (X_0 \cup \Fr_X U) - (X_0 \cap U).
\]
\end{lem}

\begin{proof}
Let $T$ be a closed-tubular neighborhood of $\alpha(D^1)$ in $X_\pm$.
There is a framing diffeomorphism $\phi: (D^1 \times D^3,\bdry D^1 \times D^3) \longra (T,T \cap X_0)$ with $\phi(s,0) = \alpha(s)$.
Define
\begin{eqnarray*}
O &:=& \phi \left\{ (s,x) \in D^1 \times D^3 \;\mid\; {\textstyle\frac{2}{3}} \leqslant \|x\| \leqslant 1 \right\}\\
M &:=& \phi \left\{ (s,x) \in D^1 \times D^3 \;\mid\; {\textstyle\frac{1}{3}} \leqslant \|x\| \leqslant {\textstyle\frac{2}{3}} \right\}\\
I &:=& \phi \left\{ (s,x) \in D^1 \times D^3 \;\mid\; 0 \leqslant \|x\| \leqslant {\textstyle\frac{1}{3}} \right\},
\end{eqnarray*}
which is a decomposition of $T = O \cup M \cup I$ into three closed subsets.
Define a map
\[
g : O \longra B_\pm ~;~ \phi(s,x) \longmapsto (f\circ\phi)(s,(3\|x\|-2) x).
\]
Since $[f \circ \alpha] = 0 \in \pi_1(B_\pm,B_0)$, there exists a map $H: D^1 \times [0,1] \longra B_\pm$ such that
\begin{eqnarray*}
H(s,1) &=& (f\circ\alpha)(s) \qquad\forall s \in D^1\\
H(\pm 1,t) &=& \alpha(\pm 1) \quad\qquad\forall t \in [0,1]\\
H^{-1}(B_0) &=& \bdry D^1 \times [0,1] \;\cup\; D^1 \times \{0\}.
\end{eqnarray*}
By the pasting lemma, we can extend $g$ from $O$ to $O \cup M$ by
\[
g: M \longra B_\pm ~;~ \phi(s,x) \longmapsto H(s,3\|x\|-1).
\]

Next, there exist both a neighborhood $C$ of the attaching 0-sphere $\alpha(\bdry D^1)$ in $X_\mp$ and a diffeomorphism $\psi: \bdry D^1 \times D^4_- \longra X_\mp$ such that $\psi|{\bdry D^1 \times D^3_0} = \phi|{\bdry D^1 \times D^3}$.
Extend $g$ from $\Fr_{X_0} T = \phi(\bdry D^1 \times \bdry D^3)$ to $\Fr_{X_\mp} C = \psi(\bdry D^1 \times \bdry_- D^4_-)$ by $g=f$.
Then $g$ is defined on the `riveted' 3-sphere $S:= \Fr_X(C\cup I)$.
Since $g(S) \subset B_\mp$ and $\pi_3(B_\mp)=0$, we may extend $g$ to the `riveted' 4-disc $C \cup I$;
using the collar of $B_0$ in $B_\mp$, we can guarantee that $g(C\cup I - S) \subset B_\mp - B_0$.
Lastly, extend $g$ to the complement $X-(C \cup T)$ by $g=f$. 
Therefore $g: X \longra B$ is transverse to $B_0$ with
\[
g^{-1}(B_0) ~=~ (X_0 - I) \;\cup\; (M \cap I).
\]

Finally, note $\Fr_X(C \cup T)$ is a 3-sphere in $X$, so we obtain a 4-sphere $X \times [0,1]$:
\[
\Sigma ~:=~ (C \cup T) \times \{0\} \;\cup\; \Fr_X(C \cup T) \times [0,1] \;\cup\; (C \cup T) \times \{1\}.
\]
Since $\pi_4(B)=0$, we may fill in $(f \circ \proj_X)|\Sigma$ to obtain a homotopy from $f$ to $g$.
\end{proof}

We adapt to dimension 4, and simplify, `arc-chasing' arguments of \cite[p67]{Hempel} and \cite[p88]{Cappell}.
Further, we generalize the sliding of 1-handles trick of Proof~\ref{lem:LWmaps}.

\begin{proof}[Proof of Proposition~\ref{prop:epi}]
By Theorem~\ref{thm:Bowditch} and by Lemma~\ref{lem:loops} with respect to the model \eqref{eqn:dmcm}, we homtope $f$ so that there are disjointly embedded 1-handlebodies $\Lambda_\pm \diffeo \# r_\pm (S^1 \times D^3) \subset X$ satisfying $f(\Lambda_\pm) \subset BG_\pm - BG_0$ and $f_\#\pi_1(\Lambda_\pm,x_\pm) = G_\pm$.
Hence $f(\Lambda_- \sqcup \Lambda_+)$ is disjoint from the bicollar neighborhood $BG_0 \times [-1,1]$ in $BG$.
Next, by Lemma~\ref{lem:transverse}, we further re-choose $f$ up to homotopy relative to $\Lambda_- \sqcup \Lambda_+$ so that $f$ is also transverse to $BG_0 \times \{0\}$, say with $f$-preimage $K$.
Write $V_\pm$ for the $X$-closure of the path-component neighborhood of $\Lambda_\pm$ in the open subset $X-K$.

Assume that $V_+ \cap K$ has at least two components, say $K_0 \sqcup K_1$.
Since $V_+$ is connected, there is a properly and smoothly embedded arc $\alpha: [0,1] \longra V_+$ satisfying: $\alpha(i) \in K_i$ if $0 \leqslant i \leqslant 1$, $\alpha(\frac{1}{2})$ is near-but-not $x_+$, and $\alpha^{-1}(\mathring{V}_+) = (0,1)$. 
Since the composite map $\pi_1(\Lambda_+) \xrightarrow{~f_\#~} \pi_1(BG_+) \longra \pi_1(BG_+,BG_0)$ is surjective, upon midpoint-concatenation of some based loop $(S^1,1) \longra (\bdry\Lambda_+,\alpha(\frac{1}{2}))$, we may assume that $[f \circ \alpha] = 0 \in \pi_1(BG_+,BG_0)$.
Then, by Lemma~\ref{lem:handleexchange}, we re-choose $f$ up to homotopy relative $\Lambda_- \sqcup \Lambda_+$ so that the new component neighborhood $V_+$ of $\Lambda_+$ in $X-f^{-1}(BG_0)$ contains $K_0 \# K_1$.
Since $X$ is compact, so is $K$, so we repeat finitely many steps until $V_+ \cap K$ becomes connected.
Similarly, make $V_- \cap K$ connected.

Write $L := V_- \cap V_+$, a connected 3-submanifold of $X$.
Let $x_0 \in L$.
Assume there exists $x_1 \in K-L$.
Since $X$ is connected, there exists a path $\gamma: [0,1] \longra X$ from $x_0$ to $x_1$.
Define $s_0 := \sup \gamma^{-1}(V_- \cup V_+)$.
Since $0 < s_0 < 1$, we must have $\gamma(s_0) \in \Fr_X(V_-) \cup \Fr_X(V_+)=L$.
Then $\gamma(s_0)$ is in the interior of $V_- \cup V_+$.
So there exists $s_1>s_0$ with $\gamma(s_1)$ also in the interior of $V_- \cup V_+$.
This contradicts the maximality of $s_0$.
Therefore $K-L$ is empty.
Hence $K=L$ and so it is connected.

Finally, since $G_0 \subset G_+$ is finitely generated and since $(f|\Lambda_+)_\# : \pi_1(\Lambda_+) \longra G_+$ is surjective, there exist based loops $\delta_1, \ldots, \delta_{r_0} : (S^1, 1) \longra (V_+,x_0)$ such that $G_0 = \gens{f_\#[\delta_1], \ldots, f_\#[\delta_{r_0}]}$.
In particular, each $f\circ\delta_i$ is based homotopic into $BG_0$.
Since each $f \circ \delta_i$ represents $0$ in $\pi_1(BG_+,BG_0)$, by Lemma~\ref{lem:handleexchange} applied $r_0$ times, we re-choose $f$ so that, further, its restriction to $M := f^{-1}(BG_0)$ is $\pi_1$-surjective.
\end{proof}

\section{Proofs of the embedding theorems}

\begin{proof}[Proof of Theorem~\ref{thm:main}]
Clearly \eqref{eqn:main} is a necessary condition.
So now, assume \eqref{eqn:main}.

Consider the double mapping cylinder model \eqref{eqn:dmcm}.
Since $G_-,G_0,G_+$ are finitely generated and $c: X \longra BG$ classifies a regular cover, by Proposition~\ref{prop:epi}, we may re-choose $c$ up to homotopy so that: $c$ is transverse to $BG_0$, the 3-submanifold preimage $M$ is connected, and the restriction $c_0: M \longra BG_0$ classifies a regular cover.
Write $X = X_- \cup_{M} X_+$ and $c = c_- \cup_{c_0} c_+$ with restrictions $c_\pm : X_\pm \longra BG_\pm$.

Next, consider the commutative square, with horizontal maps being connecting homomorphisms induced from \eqref{eqn:dmcm} and with vertical maps being of Hurewicz type:
\[\xymatrix{
\Omega^{SO}_4(BG) \ar[r]^-{\bdry} \ar[d] & \Omega^{SO}_3(BG_0) \ar[d]\\
H_4(BG) \ar[r]^-{\bdry} & H_3(BG_0)
}\qquad;\qquad
\xymatrix{
[X,c] \ar@{|->}[r] \ar@{|->}[d] & [M,c_0] \ar@{|->}[d]\\
c_*[X] \ar@{|->}[r] & c_{0*}[M].
}\]
Hence the criterion \eqref{eqn:main} implies: there is a classifying map $d: X_0 \longra BG_0$ with
\[
d_*[X_0] ~=~ c_{0*}[M] ~\in~ H_3(G_0).
\]
Since $d_\#$ and $c_{0\#}$ are surjective, by Lemma~\ref{lem:LWmaps}, there is a 4-dimensional oriented smooth bordism $e: V \longra BG_0$ from $(M,c_0)$ to $(X_0,d)$ made with only 2-handles.
Since $V$ is obtained from $X_0$ using only 2-handles, the inclusion $j: X_0 \longra V$ induces an epimorphism on fundamental groups.
Since $d_\# = e_\# \circ j_\#$ is a monomorphism, note that $j_\#$ is also a monomorphism.
So both $j_\#$ and $e_\#$ are isomorphisms.

Now, we obtain a connected 5-dimensional oriented compact smooth cobordism
\[
T ~:=~ X \times [0,1] \;\cup_{M \times [-1,1]}\; V \times [-1,1]
\]
where we regard $M \times [-1,1]$ in $X \times \{1\}$ and we smooth the corners at $M \times \{-1,1\}$.
The resultant 4-manifold and map are $X' := \bdry T - X \times \{0\}$ and $c' := D |_{X'}$, where
\[
D: T \longra BG ~;~
\begin{cases}
(x,t) \in X \times [0,1] &\longmapsto c(x)\\
(v,s) \in V \times [-1,1] &\longmapsto (e(v),s).
\end{cases}
\]
Decompose the space $X' = X'_- \cup_{X_0} X'_+$ with $X'_\pm = X_\pm \cup_{M} V \times \{\pm 1\} \cup X_0 \times [\pm 1, 0]$, as well as the map $c' = c'_- \cup_{d} c'_+: X' \longra BG$ with $c'_\pm = c_\pm \cup_{c_0} e: X'_\pm \longra BG_\pm$.

Write $i: M \longra X$ for the inclusion.
Since $V$ is the trace of a surgery on a framed oriented link $L$ in $M$, correspondingly note $T$ is the trace of a surgery on $i \circ L$ in $X$.
By a similar argument as earlier, we find that the kernel of $c_{0\#}$ equals the kernel of the map induced by the inclusion $M \longra V$, which is generated by the (unbased) components $L_k$ of $L$, upon anchoring them to the basepoint with choices of connecting paths.
In addition, since $c_{0\#} = c_\# \circ i_\#$ and $c_\#$ is an isomorphism, the kernel of $c_{0\#}$ equals the kernel of $i_\#$.
In particular, each embedded circle $L_k$ is nulhomotopic in $X$, bounding an immersed disc with tranverse double points, which can be isotoped away using finger-moves \cite[1.5]{FQ}; thus each $L_k$ bounds an embedded disc in $X$.
Another consequence is that $D \simeq c \cup_{c_0} e: X \cup_M V \longra BG$ induces an isomorphism on fundamental groups.
Then, by an argument with alternating words, each $c'_\pm: X'_\pm \longra BG_\pm$ also does so.
So, since $d_\#$ is an isomorphism, $c'$ induces an isomorphism on fundamental groups.

Finally, we show that the embedding solution $X'$ is bistably diffeomorphic to $X$.
For each $L_k$, consider embedded in $\mathring{T}$ the 2-sphere $S_k$ with equator $L_k$, with northern hemisphere the core of the bounding 2-handle in $V \times \{0\}$, and with southern hemisphere the bounding 2-disc in $X \times \{1\}$.
Write $N_k$ for the 5-dimensional closed-tubular neighborhood of $S_k$ in $\mathring{T}$.
Observe that $T$ is diffeomorphic to the boundary-connected sum $(X \times [0,1]) \natural (\bigsqcup_k N_k)$.
Each $N_k$ is diffeomorphic to either $D^3 \times S^2$ or $D^3 \rtimes S^2$, where the latter is the nontrivial (nonspin) disc bundle.
Thus, we obtain $X' \diffeo X \# p(S^2 \times S^2) \# q(S^2 \rtimes S^2)$ for some $p \geqslant r$ and $q \geqslant 0$.
Since $(S^2 \times S^2) \# (\C P_2) \diffeo 2 (\C P_2) \# (\ol{\C P_2})$ and $S^2 \rtimes S^2 \diffeo (\C P_2) \# (\ol{\C P_2})$ \cite[C1, L1]{Wall},
\[
X'(1,0) ~\diffeo~ X(1+p+q,p+q).\qedhere
\]
\end{proof}

\begin{proof}[Proof of Theorem~\ref{thm:totallynonspin}]
Since $w_2(\wt{X}) \neq 0$, by the Hurewicz theorem, there exists a spherical class $\wt{\alpha}: S^2 \longra \wt{X}$ such that $\Kronecker{w_2(\wt{X})}{\wt{\alpha}_*[S^2]} \neq 0$.
Write $p: \wt{X} \longra X$ for the covering map, and write $\alpha := p \circ \wt{\alpha}: S^2 \longra X$.
Since on tangent bundles $T\wt{X} = p^*(TX)$, as one obtains the smooth structure on $\wt{X}$ by even-covering, note
\[
\Kronecker{w_2(X)}{\alpha_*[S^2]} = \Kronecker{w_2(X)}{p_*\wt{\alpha}_*[S^2]} = \Kronecker{p^* w_2(X)}{\wt{\alpha}_*[S^2]} = \Kronecker{w_2(\wt{X})}{\wt{\alpha}_*[S^2]} = 1.
\]

Do the same as in the proof of Theorem~\ref{thm:main}, until the construction of the 2-sphere $S_k$.
In the case that the normal bundle of $S_k$ is nontrivial, replace the southern hemisphere with its one-point union with $\alpha$, smoothed rel $L_k$ into immersion then an embedding by finger-moves, to obtain $S_k'$.
Since $[S_k'] = [S_k] + [\alpha] \in \pi_2(X)$, note
\[
\Kronecker{w_2(X)}{S'_{k*}[S^2]} = \Kronecker{w_2(X)}{S_{k*}[S^2] + \alpha_*[S^2]} = 1 + 1 = 0 ~\in~ \Z/2.
\]
Hence the normal 3-plane bundle of the new embedded 2-sphere $S_k'$ in $\mathring{T}$ is trivial.
So $T$ is diffeomorphic to the boundary-connected sum $(X \times [0,1]) \natural (\bigsqcup_{k=1}^r D^3 \times S^2)$.
Therefore, we obtain $X'$ is diffeomorphic to $X(r) = X \# r(S^2 \times S^2)$.
\end{proof}

\begin{proof}[Proof of Theorem~\ref{thm:spin}]
Do the same as in the proof of Theorem~\ref{thm:main}, except using \eqref{eqn:spin} and Lemma~\ref{lem:LWmaps_spin} instead of \eqref{eqn:main} and Lemma~\ref{lem:LWmaps}, until the construction of the 2-sphere $S_k$.
Here, the spin structure $s_M$ on $M$ is the restriction of the spin structure $s$ on $X \times \{1\}$, where the spin structure on the normal line bundle is induced from its pullback orientation \cite[II:2.15]{LM}.
Since $s_M$ is the restriction of the spin structure on $V \times \{0\}$, we obtain that $T$ has an induced spin structure.

Then, since $w_2(T)=0$, the normal 3-plane bundle of each $S_k$ in $\mathring{T}$ is trivial.
So $T$ is diffeomorphic to the boundary-connected sum $(X \times [0,1]) \natural (\bigsqcup_{k=1}^r D^3 \times S^2)$.
Therefore, we obtain $X'$ is diffeomorphic to $X(r) = X \# r(S^2 \times S^2)$.
\end{proof}

For clarity, we repeat the following proof from \cite[p258]{KLT} and \cite[p713]{Kreck}.
The statement shall be applied in Proof~\ref{thm:prespin} for manifolds $Y$ of dimensions 3, 4, 5.

\begin{proof}[Proof of Lemma~\ref{lem:prespin}]
Since $\wt{Y}$ is 1-connected, by the Leray--Serre spectral sequence for the homotopy fibration sequence $\wt{Y} \xrightarrow{~p~} Y \xrightarrow{~u~} B\G$, we obtain an exact sequence
\begin{equation}\label{eqn:Kreck}
\xymatrix{
0 \ar[r] & H^2(B\G;\Z/2) \ar[r]^-{u^*} & H^2(Y;\Z/2) \ar[r]^-{p^*} & H^2(\wt{Y};\Z/2)^G.
}
\end{equation}
Then, since $w_2(T\wt{Y})=w_2(p^*(TY)) = p^*(w_2(TY))$, the oriented smooth manifold $\wt{Y}$ admits a spin structure if and only if there exists $w_2^u \in H^2(B\G;\Z/2)$ such that $u^*(w_2^u)=w_2(TY)$.
Further by exactness, this class $w_2^u$ is unique if it exists.
\end{proof}

For $r \geqslant 0$, the \emph{pinch map} $p: X(r) \longra X \vee \# r(S^2 \times S^2)$ gives a degree-one map
\[
k := (\id \vee \const) \circ p : X(r) \longra X.
\]
The $\pi_1$-isomorphism $c: X \longra BG$ induces the $\pi_1$-isomorphism $c \circ k: X(r) \longra BG$.

\begin{proof}[Proof of Theorem~\ref{thm:prespin}: necessity of \eqref{eqn:prespin}]
Assume for some $r \geqslant 0$ that there exists an incompressible embedding $j_0: X_0 \longra X(r)$ such that $(c \circ k \circ j_0)_\#(\pi_1 X_0) = G_0$.
Then $X(r) = X_-' \cup_{X_0} X_+'$ with inclusions $j_\pm : X_\pm' \longra X(r)$.
Since $X(r)$ and $X_0$ are connected, so are $X_\pm'$.
Furthermore, since $(c \circ k)_\#$ and $(c \circ k \circ j_0)_\#$ are isomorphisms, by a basic observation on normal form \cite{Smith_dissertation}, so are $(c \circ k \circ j_\pm)_\#: \pi_1(X_\pm') \longra G_\pm$.

Consider the double mapping cylinder model \eqref{eqn:dmcm} of $BG$, where $Bi_0: BG_0 \longra BG$ is the inclusion of a bicollared subspace.
Since $X_0$ is a CW-complex, there is a homotopically unique map $d: X_0 \longra BG_0$ such that $Bi_0 \circ d \simeq c \circ k \circ j_0$.
Furthermore, since $X_\pm$ are CW-complexes, $d$ extends to maps $c_\pm: X_\pm' \longra BG_\pm \cup BG_0 \times [0, \pm 1]$ with $Bi_\pm \circ c_\pm \simeq c \circ k \circ j_\pm$.
Therefore, $c\circ k$ is homotopic to a $BG_0$-transverse map $c' := c_-' \cup_d c_+' : X(r) \longra BG$ satisfying $(c')^{-1}(BG_0 \times \{0\}) = X_0$.

Next, since $\wt{X}$ admits a spin structure, by Lemma~\ref{lem:prespin}, there is a unique class $w_2^c \in H^2(BG;\Z/2)$ such that $w_2(TX) = c^*(w_2^c)$.
Since $S^2$ is stably parallelizable, so is $S^2 \times S^2$.
Then the tangent bundle $TX(r)$ is stably isomorphic to the pullback $k^*TX$. (The corresponding statement is false for a bistabilization $X(a,b)$ unless $a=0=b$.)
Hence $w_2(TX(r))=k^*w_2(TX)$.
Note
\begin{eqnarray*}
d^*(i_0^* w_2^c) &=& (i_0 \circ d)^*(w_2^c) ~=~ (c \circ k \circ j_0)^*(w_2^c) ~=~ j_0^* k^*(c^* w_2^c)\\
&=& j_0^* k^*(w_2(TX)) ~=~ j_0^* w_2(TX(r)) ~=~ w_2(TX_0) ~=~ v_2(X_0) ~=~ 0,
\end{eqnarray*}
with $v_2 = w_2 + w_1^2$ the second Wu class \cite[11.14]{MS} and $\Sq^2=0$ on $H^1(X_0;\Z/2)$.
The exact sequence \eqref{eqn:Kreck} holds analogously for $X_0$, so $\Ker(d^*)=0$ hence $i_0^*(w_2^c)=0$.
Thus, there is a nulhomotopy $\theta$ of $w_2^c \circ Bi_0$.
By Lemma~\ref{lem:transverse}, we may assume $c$ is transverse to $BG_0$ in model \eqref{eqn:dmcm}, with 3-submanifold $M := c^{-1}(BG_0 \times \{0\})$ of $X$.

Now, $k: X(r) \longra X$ extends to a retraction $K: X[r] \simeq X \vee r S^2 \longra X$, where $X[r] := (X \times [0,1]) \;\natural\; r(D^3 \times S^2)$ is the canonical cobordism from $X$ to $X(r)$.
Since $c$ is transverse to $BG_0$, so is $c \circ K$.
Recall there is homotopy $H: X(r) \times [1,2] \longra BG$ such that $H(-,1)=c \circ k$ and $H(-,2)=c'$.
These unite to a $B_0$-transverse map
\[
C ~:=~ (c\circ K) \cup_{c \circ k} H ~:~ W := X[r] \cup (X(r) \times [1,2]) \longra BG.
\]
The preimage 4-manifold $V := C^{-1}(BG_0 \times \{0\})$ fits into an oriented bordism $(V,C|V)$ from $(M,c|M)$ to $(X_0,d)$.
Furthermore, this enhances to a spin bordism, as Definition~\ref{defn:inducedspin} produces a spin structure $s_\mu^\theta$ on $V$ defined by the formula
\[
s_\mu^\theta : V \longra B\Spin = \hofib(w_2) ~;~ x \longmapsto \left( \tau_V(x), \mu^x * \theta^{C(x)} \right),
\]
with $\mu: W \times [0,1] \longra K(\Z/2,2)$ a homotopy from $w_2 \circ \tau_W$ to $w_2^C \circ C$.
Indeed, $w_2^C$ exists by Lemma~\ref{lem:prespin}, since $\wt{W}$ has a spin structure as $T\wt{W} \iso \wt{K}^* T\wt{X} \oplus \ul{\R}$.
Define $\eta: X \times [0,1] \longra K(\Z/2,2)$ as a restriction of $\mu$.
So $s_\mu^\theta$ on $V$ restricts to spin structures $s_\eta^\theta$ on $M$ and $t := s_\mu^\theta | X_0$ on $X_0$.
Therefore, Equation~\eqref{eqn:prespin} holds.
\end{proof}

Recall the \textbf{homotopy fiber} of a map $f: A \longra B$ with respect to $b_0 \in B$ is
\[
\hofib(f) ~:=~ \{ (a,p) \in A \times B^{[0,1]} \;\mid\; p(0)=f(a) \text{ and } p(1)=b_0 \}.
\]

\begin{proof}[Proof of Theorem~\ref{thm:prespin}: sufficiency of \eqref{eqn:prespin}]
Assume Equation~\eqref{eqn:prespin} holds, where the transverse 3-submanifold $M := c^{-1}(BG_0)$ of $X$ exists by Lemma~\ref{lem:transverse}, upon altering $c$ by a homotopy.
Furthermore, by Proposition~\ref{prop:epi}, we can further homotope $c$ so that $M$ is connected and its restriction $c_0: M \longra BG_0$ is a $\pi_1$-epimorphism.
Then the spin bordism $(V,\Sigma,e)$ from $(M,s_\eta^\theta,c|M)$ to $(X_0,t,d)$, by Lemma~\ref{lem:LWmaps_spin}, can be assumed to only have 2-handles relative to $X_0$.
From the proof of Theorem~\ref{thm:main}, $e: V \longra BG_0$ is a $\pi_1$-isomorphism, and the map $D \simeq c \cup_{c_0} e: T \longra BG$ is also.

Observe that the spin structure $\Sigma: V \longra B\Spin=\hofib(w_2)$ is of the form
\[
\Sigma ~=~ \left( \tau_V: V \longra BSO, \; \sigma: V \longra K(\Z/2,2)^{[0,1]} \right),
\]
with $\sigma(x) \in K(\Z/2,2)^{[0,1]}$ a path from $w_2(\tau_V(x))$ to the basepoint $\omega$ of $K(\Z/2,2)$.
Recall that $\theta : BG_0 \times [0,1] \longra K(\Z/2,2)$ is a homotopy from $w_2^c \circ Bi_0$ to $\const_\omega$.
Then define a homotopy $\xi: V \times [0,1] \longra K(\Z/2,2)$ from $w_2 \circ \tau_V$ to $w_2^c \circ Bi_0 \circ e$ by
\[
\xi^x ~:=~ \sigma(x) * \overline{\theta^{e(x)}}.
\]
Recall that $\eta: X \times [0,1] \longra K(\Z/2,2)$ is a homotopy from $w_2 \circ \tau_X$ to $w_2^c \circ c$.
This restricts to a homotopy $\eta_0: M \times [0,1] \longra K(\Z/2,2)$ from $w_2 \circ \tau_M$ to $w_2^c \circ Bi_0 \circ c_0$.
Note $\xi$ extends $\eta_0$, since $\tau_V$ extends $\tau_M$ and $e$ extends $c_0$.
Thus, since $T \simeq X \cup_M V$, we obtain a homotopy $\eta \cup_{\eta_0} \xi$ from $w_2 \circ \tau_T$ to $w_2^c \circ D$.
Since $D$ classifies the universal cover of the 5-manifold $T$, by Lemma~\ref{lem:prespin}, the universal cover $\wt{T}$ has a spin structure.

Consider the embedded 2-spheres $S_k : S^2 \longra \mathring{T}$, in the proof of Theorem~\ref{thm:main}.
Write $P: \wt{T} \longra T$ for the universal covering map.
As $S^2$ is simply connected, by the lifting theorem, there is an embedding $\wt{S_k}: S^2 \longra \wt{T}$ with $S_k = P \circ \wt{S_k}$.
Note
\[
\Kronecker{w_2 T}{S_{k*}[S^2]} = \Kronecker{w_2 T}{P_* \wt{S_k}_*[S^2]} = \Kronecker{P^* (w_2 T)}{\wt{S_k}_*[S^2]} = \Kronecker{w_2 \wt{T}}{\wt{S_k}_*[S^2]} = 0.
\]
Then, although $T$ need not be spin, nonetheless the normal 3-plane bundle of each $S_k$ in $\mathring{T}$ is trivial.
So $T$ is diffeomorphic to the boundary-connected sum $(X \times [0,1]) \natural (\bigsqcup_{k=1}^r D^3 \times S^2)$.
Therefore $X'$ is diffeomorphic to $X(r) = X \# r(S^2 \times S^2)$.
\end{proof}

A final remark on \eqref{eqn:main} is that $H_3(G_0)=\Z$ if $X_0$ is irreducible with infinite fundamental group, as $X_0$ models $BG_0$, a consequence of the sphere theorem \cite{Hempel}.

\subsection*{Acknowledgements}

The second author is a doctoral student of the first author and thanks him for equal involvement on this stable existence project.
His dissertation solves the stable uniqueness problem, using instead the power of Kreck's modified surgery machine \cite{Smith_dissertation}.
The first author is grateful to his low-dimensional topology professors at U Illinois--Chicago, Louis Kauffman and Peter Shalen, who instilled in him the beauty of knots and of incompressible surfaces in 3-manifolds.
This project is a `boyhood dream' to see if that philosophy works up one dimension.

\bibliographystyle{alpha}
\bibliography{StableExistence_Cutting4manifolds}

\end{document}